\documentclass[a4paper,10pt]{amsart}
\newcommand{\K}[1]{\mathbb{#1}}

\newcommand{\PP}{\mathbb{P}}

\usepackage{amsmath}
\usepackage{amsthm}
\usepackage{amssymb}
\usepackage[a4paper, bottom=5.33cm]{geometry}
\usepackage{graphicx}
\usepackage{microtype}
\usepackage{enumitem}
\usepackage{color}
\usepackage{relsize}
\usepackage{mathrsfs}
\usepackage{cleveref}

\usepackage[utf8]{inputenc}
\usepackage[russian,english]{babel}
\usepackage{mathtools}

\newtheorem{thm}{Theorem}
\newtheorem*{thm*}{Theorem}
\newtheorem{propo}[thm]{Proposition}
\newtheorem{lemma}[thm]{Lemma}
\newtheorem{cor}[thm]{Corollary}
\newtheorem{conjecture}{Conjecture}
\newtheorem{problem}{Problem}

\theoremstyle{definition}

\newtheorem{example}{Example}

\title[Rational functions with only real periodic points]{Rational functions with only real periodic points}

\author{Khazhgali Kozhasov}
\address{Technische Universit\"at Braunschweig, 38106 Braunschweig, Germany} 
\email{k.kozhasov@tu-braunschweig.de}

\author{Mario Kummer}
\address{Technische Universit\"at Berlin, D-10623 Berlin, Germany} 
\email{kummer@tu-berlin.de}
\thanks{Khazhgali Kozhasov, Technische Universit\"at Braunschweig, Germany (k.kozhasov@tu-braunschweig.de);}

\thanks{Mario Kummer (corresponding author), Technische Universit\"at Berlin, Germany (kummer@tu-berlin.de).}

\thanks{Mario Kummer has been supported by the DFG under Grant No.421473641.}

\newcommand{\comment}[1]{}

\begin{document}

\subjclass[2010]{Primary: 37F10, 26C10, 26C15; Keywords: Periodic points, rational functions, Julia set, interlacing, Chebyshev polynomials.}

\begin{abstract}
We study self-morphisms of smooth real projective algebraic curves that have only real periodic points. In the case of the projective line we provide a convenient characterization of such morphisms. We derive a semialgebraic description of the component of real fibered rational functions all of whose periodic points are real.
\end{abstract}
\maketitle

\section*{Introduction}

The study of iterates of rational functions takes a prominent place in discrete dynamical systems having strong connections to iterative methods, complex and arithmetic geometry. A well-known example is Newton's method \cite{HSS} that is used to find an approximate root of a polynomial $p\in \K{C}[z]$ by applying iterates of the rational function $f(z)=z-\frac{p(z)}{p^\prime(z)}$ to some initial guess $z_0\in \K{C}$. However, for particular $z_0$ the Newton's method fails to converge to a root of $p$. This happens, for example, when $z_0$ is a periodic point of $f$, that is, $z_0$ returns to itself under sufficiently many applications of $f$.

Any rational function $f\in \K{C}(z)_d$ of degree $d\geq 2$ is known to have infinitely many periodic points in $\K{C}$ \cite{Kalantari}. Northcott observed \cite{Northcott} that if $f\in \K{Q}(z)_d$ is defined over the field of rational numbers, then it can have only finitely many periodic points in $\K{Q}$. One of the main open problems in arithmetic dynamics is the uniform boundedness conjecture \cite{MS} asserting that the number of rational periodic points of $f\in \K{Q}(z)_d$ is uniformly bounded by a constant depending only on the degree $d$ of $f$. Remarkably, this problem remains open even for $d=2$, see \cite{survey}. When the field is relaxed to $\K{R}$ both of the above ``regimes'' are possible: some real rational functions $f\in \K{R}(z)_d$ of degree $d\geq 2$ have infinitely many periodic points in $\K{R}$ and some have only finitely many. Moreover, there exist rational functions of any given degree whose all periodic points are real, see \cite[\S 16, 17]{fatou} and \cite[Exp.~2.2]{dynamics}.

A function $f\in \K{R}(z)_d$ has only real periodic points if, by definition, for all $k\geq 1$ the equation $f^k(z)=z$ has only real solutions $z\in \overline{\K{R}}$, where $f^k$ denotes the $k$th iterate of $f$. In Theorem \ref{thm:main} we provide an equivalent and potentially easier way to certify this property. In this result we exploit a characterization of rational functions with real Julia sets by Eremenko and van Strien \cite{realmult}. Real fibered rational functions constitute  a distinguished class of such functions, they map only real points to real points. Dynamical properties of this class of functions were studied by Fatou \cite{fatou}. Using his results, in Theorem \ref{thm:RF} we provide a semialgebraic description of the component of real fibered rational functions that have only real periodic points. To illustrate Theorem \ref{thm:main}, in Corollaries \ref{cor_odd} and \ref{cor_even} we give sufficient conditions for a polynomial to have only real periodic points, with Chebyshev polynomials being a classical example. We conjecture that this property is also shared by the classical Hermite polynomials, see Example \ref{Hermite} and Conjecture \ref{conj}.

Rational functions are morphisms of the projective line. In Propositions \ref{g=1} and \ref{g>1} we study self-morphisms of a real algebraic curve of higher genus, it turns out that realness of periodic points can be understood rather easily in this case. Another generalization would be to consider self-morphisms of the projective space of higher dimension and study realness of their periodic points, see Problem \ref{problem}. 
\section{Main results}

\underline{A rational function of degree $d$} is a ratio $f=\frac{p}{q}$ of two polynomials $p=p_dz^d+\dots+p_1z+p_0$, $q=q_dz^d+\dots +q_1z+q_0\in \K{C}[z]$ that have no common factors and such that $p_d$ and $q_d$ are not simultaneously zero. The polynomials $p$ and $q$ are defined uniquely up to a non-zero common factor. Thus, the set $\K{C}(z)_{d}$ of rational functions of degree $d$ is identified with the complement of a hypersurface in \underline{the projective $(2d+1)$-space} $\PP^{2d+1}=\{(p_d,\dots,p_1,p_0,q_d,\dots,q_1,q_0)\in \K{C}^{2d+2}\setminus \{0\}\}/\K{C}^*$. 
Any rational function $f=\frac{p}{q}$ defines a continuous map $z\mapsto \frac{p(z)}{q(z)}$ from \underline{the Riemann sphere} $\overline{\K{C}} =\K{C}\cup \{\infty\}$ into itself that sends $\infty$ to $\lim_{\,z\rightarrow \infty} \frac{p(z)}{q(z)} \in \overline{\K{C}}$ and zeros of $q$ to $\infty$.

A point $z\in \overline{\K{C}}$ is called \underline{periodic} for a rational function $f$ if $f^k(z)=z$ for some integer $k\geq 1$, where one defines $f^1=f$ and $f^{k+1}=f^k\circ f$. In particular, if $f(z)=z$, then $z\in \overline{\K{C}}$ is called a \underline{fixed point} of $f$. Thus, a periodic point of $f$ is a fixed point of $f^k$ for some $k\geq 1$ and the smallest such $k$ is called the \underline{period} of $z$. If $z\in \overline{\K{C}}$ is a periodic point of $f\in \K{C}(z)_d$ of period $k$, the \underline{multiplier} of $z$ is defined to be  $(f^k)^\prime(z)$ if $z\in \K{C}$ and $\left(\lim_{\,w\rightarrow \infty} (f^k)^{\prime}(w)\right)^{-1}$ if $z=\infty$. The multiplier of any point in the cycle $\{z,f(z),\dots,f^{k-1}(z)\}$ of $z$ is the same (see, e.g., \cite[\S 4]{milnor}). A periodic point $z\in \overline{\K{C}}$ is called \underline{attracting}, \underline{indifferent} or \underline{repelling} if its multiplier $\lambda_z$ satisfies $\vert\lambda_z\vert<1$, $\vert \lambda_z\vert=1$ or $\vert\lambda_z\vert>1$ respectively. Note that when $f\in \K{R}(z)_d$ is a real rational function the multiplier of every real periodic point of $f$ is a real number. The \underline{Julia set} $J(f)$ of $f\in \K{C}(z)_d$ is defined to be the closure of the set of repelling periodic points of $f$.

In \cite{realmult} Eremenko and van Strien characterized real rational functions $f\in \K{R}(z)_d$ whose Julia set is real, i.e., $J(f)\subset \overline{\K{R}}$, proving the following result.

\begin{thm}[\cite{realmult}]\label{thm:realjulia}
 Let $f\in\K{R}(z)_d$ be a rational function whose Julia set is real. Then there are the following possibilities:
 \begin{itemize}
  \item[(i)] $f$ is real fibered in the sense that $z\in\overline{\K{R}}\Leftrightarrow f(z)\in\overline{\K{R}}$ for all $z\in\overline{\K{C}}$.
 \end{itemize}
 If $(i)$ does not hold, then there is a real critical point and a fixed point $x_0\in\K{R}$ with $f'(x_0)\in[-1,1]$. Let $I\subset\overline{\K{R}}$ be the smallest closed interval which contains $J(f)$ and whose interior does not contain $x_0$. Then one of the following holds:
 \begin{itemize}
  \item[(ii)] $I\subsetneq\overline{\K{R}}$ and $z\in I\Leftrightarrow f(z)\in I$ for all $z\in\overline{\K{C}}$.
  \item[(iii)] $f(I)$ strictly contains $I$.
 \end{itemize}
\end{thm}

Using this result we characterize rational functions that have only real periodic points.

\begin{thm}\label{thm:main}
 Let $f\in\K{R}(z)_d$ with $d\geq2$. Then the following are equivalent:
 \begin{enumerate}
  \item $f$ has only real periodic points;
  \item there is a nonrepelling real cycle of length at most $2$ and a closed semialgebraic subset $S\subset\overline{\K{R}}$ that contains a nonattracting fixed point such that $f^{-1}(S)\subset S$.
 \end{enumerate}
 If $(1)$ and $(2)$ hold, then $S$ contains all nonattracting periodic points of $f$.
\end{thm}

In the following let us denote by $\mathscr{R}_d$ the set of $f\in \K{R}(z)_d$ with only real periodic points.
We now derive several corollaries of Theorem \ref{thm:main}. Let us first consider an example of a real fibered rational function in $\mathscr{R}_d$.

\begin{example}\label{RF}
 Let $f=\frac{p}{q}\in\K{R}(z)_d$, $d\geq2$, with $p$ and $q$ polynomials with strictly interlacing zeros. Assume that $f$ has an attracting real fixed point. Then $f$ has only real fixed points by \Cref{thm:main}. The same is true for any $g$ in a  small open neighbourhood $U\subset\K{R}(z)_d$ of $f$. Thus $f$ is in the interior of $\mathscr{R}_d$. We obtain such an $f$, for example, if $\deg(p)>\deg(q)$ and the leading coefficient of $p$ is larger than the one of $q$. In that case $f$ has an attracting fixed point at infinity.
\end{example}

In fact the only real fibered rational functions that lie in the interior of $\mathscr{R}_d$ are those described in Example \ref{RF} as we show in the following Theorem.

\begin{thm}\label{thm:RF}
A real fibered rational function $f\in \K{R}(z)_d$, $d\geq 2$, is in $\mathscr{R}_d$ if and only if $f^2$ has only real fixed points. In particular, such functions form a closed semialgebraic set in $\mathscr{R}_d\subset \K{R}(z)_d$ and its boundary consists of real fibered $f\in \K{R}(z)_d$ with an indifferent fixed point.
\end{thm}

Next we would like to give sufficient conditions for a polynomial to be in $\mathscr{R}_d$.
\begin{cor}\label{cor_odd}
 Let $f\in\K{R}[x]$ be a polynomial of odd degree $2m+1\geq3$ with positive leading coefficient such that $f'$ has only real and simple zeros $z_1<\cdots< z_{2m}$. Let $x_0<z_1$ and $x_1>z_{2m}$ be fixed points of $f$ such that $f(z_{2i})\leq x_0$ and $f(z_{2i-1})\geq x_1$ for all $i=1,\ldots,m$. Then $f$ has only real periodic points.
\end{cor}

\begin{proof}
  For any $y\in[x_0,x_1]$ the polynomial $f(x)-y$ has $2m+1$ real zeros in the interval $[x_0,x_1]$ by the intermediate value theorem. Thus $f^{-1}([x_0,x_1])\subset[x_0,x_1]$. Furthermore, $f$ has an attracting fixed point at infinity. Finally, the fixed points $x_0, x_1\in [x_0,x_1]$ are nonattracting and the claim follows from Theorem \ref{thm:main}.
\end{proof}
A similar criterion holds for polynomials of even degree.

\begin{cor}\label{cor_even}
  Let $f\in \K{R}[x]$ be a polynomial of even degree $2m\geq 2$ with positive leading coefficient such that the derivative $f^\prime$ has only real and simple zeros $z_1<\dots<z_{2m-1}$. Let $x_0<z_1$ be such that $f(z_{2i-1})\leq x_0$ for all $i=1,\dots,m$ and let $x_1>z_{2m-1}$ be a fixed point of $f$ such that $f(z_{2i})\geq x_1$ for all $i=1,\dots,m-1$. Then $f$ has only real periodic points. 
\end{cor}

\begin{proof}
  For any $y\in [x_0,x_1]$ the polynomial $f(x)-y$ has $2m$ real zeros in $[x_0,x_1]$ by the intermediate value theorem and hence $f^{-1}([x_0,x_1])\subset [x_0,x_1]$. The claim follows from Theorem \ref{thm:main}, since $f$ has an attracting fixed point at infinity and the fixed point $x_1\in [x_0,x_1]$ is nonattracting.
\end{proof}

\begin{example}\label{Cheb}
The \underline{Chebyshev polynomials of the first kind} defined by $T_d(z) = \cos(d\arccos z)$, $z\in [-1,1]$, are known to have only real periodic points, see \cite[Exp.~2.2]{dynamics}. This also follows from Corollaries \ref{cor_odd} and \ref{cor_even} with $x_0=-1$, $x_1=1$. % \textbf{Show that $T_d$ is not in the interior of $\mathscr{R}_d$ although no neutral periodic points, i.e., $T_d^k(x)=x$, $k\geq 1$, has only simple solutions.} 
\end{example}

\begin{example}
 Unlike in the case of a real fibered rational function, a polynomial $f$ can lie on the boundary of $\mathscr{R}_d$ although the equation $f^k(z)=z$ has only simple real solutions for all $k$. Indeed, let $f(z)=T_2(z)=2z^2-1$ be the Chebyshev polynomials of degree $2$. Then $f^k=T_{d^k}$ has only real periodic points none of which is indifferent. On the other hand, for all sufficiently small $\varepsilon>0$ we consider $f_\varepsilon(z)=(2-\varepsilon)z^2+\varepsilon-1$. Since $1$ is a repelling fixed point of $f_\varepsilon$, it is contained in the Julia set. Since the Julia set is backward closed and $f_{\varepsilon}(-1)=1$, also $-1$ is in the Julia set. By the same reason $f^{-1}(-1)\nsubseteq\K{R}$ is contained in the Julia set. Therefore, there must be nonreal (repelling) periodic points of $f_\varepsilon$. This shows that $f=f_0$ is on the boundary of $\mathscr{R}_d$.
\end{example}

The \underline{Hermite polynomials} defined by $H_d(z)=(-1)^de^{z^2}\frac{\mathrm{d}^d}{\mathrm{d}z^d} e^{-z^2}$ seem to also have only real periodic points as the following example shows.
\begin{example}\label{Hermite}
  Consider the cubic Hermite polynomial $H_3(z)=8z^3-12z$. The zeros of $H_3'(z)$ are $-\frac{1}{2}\sqrt{2},\frac{1}{2}\sqrt{2}$. The critical values of $H_3$ are $4\sqrt{2},-4\sqrt{2}$ respectively. Finally $-\sqrt{\frac{13}{8}}, \sqrt{\frac{13}{8}}$ are fixed points that satisfy the assumptions of Corollary \ref{cor_odd}. Thus $H_3$ has only real periodic points.
  Similarly, the quartic Hermite polynomial $H_4(z)=16 z^4- 48 z^2 +12$ has only real periodic points. The value of $H_4$ at its two local minima $-\sqrt{\frac{3}{2}}, \sqrt{\frac{3}{2}}$ is $-24$, its value at the local maximum $0$ is $12$ and the fixed points $x_0\approx -1.66327$, $x_1\approx 1.66327$ satisfy the assumptions of Corollary \ref{cor_even}.   
\end{example}

This example and some computer experiments motivate the following conjecture.
\begin{conjecture}\label{conj}
For any $d\geq 1$ the Hermite polynomial $H_d\in \K{R}[x]_d$ has only real periodic points, that is, $H_d\in \mathscr{R}_d$.
\end{conjecture}

A rational function $f\in \K{C}(x)_d$ is the same thing as a degree $d$ morphism of $\PP^1\simeq \overline{\K{C}}$ into itself. More generally, one can look at morphisms $f:\PP^n \rightarrow \PP^n$ of the projective $n$-space $\PP^n=(\K{C}^{n+1}\setminus\{0\})/\K{C}^*$ into itself. The study of dynamical properties of such higher-dimensional maps was pioneered by Fornaess and Sibony in \cite{FS}. As in the case $n=1$, any morphism $f:\PP^n\rightarrow \PP^n$ of degree $d\geq 2$ has infinitely many periodic points in $\PP^n$ \cite[Thm. $3.3$]{FS}. In \cite{ASS} Abo, Seigal and Sturmfels conjectured that for any $d\geq 2$ and $n\geq 2$ there exists a real morphism $f:\PP^n\rightarrow \PP^n$ of degree $d$ that has only real fixed points. In \cite{Kozhasov2018} the first author of the present work confirmed this conjecture. It is natural to ask whether the same is true for periodic points.
\begin{problem}\label{problem}
  Let $d\geq 2$ and $n\geq 1$. Does there exist a morphism $f:\PP^n\rightarrow \PP^n$ of degree $d$ that is defined over $\K{R}$ and has only real periodic points?  
\end{problem}
Examples \ref{RF} and \ref{Cheb} imply that the answer to this question is positive for $n=1$. Note that for $n\geq 2$ the only real fibered morphisms $f:\PP^n\rightarrow \PP^n$, i.e., those that satisfy $f^{-1}(x)\subset \PP^{n}(\K{R})$ for any $x\in \PP^{n}(\K{R})$, are projective linear transformations \cite[Cor. 2.20]{realfib}. Therefore, no extension of the construction from Example \ref{RF} to $n\geq 2$ can exist. Also, all morphisms $f: \PP^n\rightarrow \PP^n$ with only real fixed points from \cite{Kozhasov2018} that we tested turn out to have many non-real periodic points.

% \textbf{Using Corollaries \ref{cor_odd} and \ref{cor_even} we should be able to prove that all Hermite polynomials have only real periodic points. Are they in the interior of $\mathscr{R}_d$?}

\section*{Higher genus}
Let $X$ be a smooth irreducible projective real algebraic curve of genus $g>0$. We conclude with noting that the question of which morphisms $X\to X$ have only real periodic points is rather simple in this situation.

\begin{propo}\label{g=1}
 Let $g=1$ and assume that the real, nonconstant morphism $f:X\to X$ has no nonreal periodic points. Then $f$ has no periodic points.
\end{propo}

\begin{proof}
 We can write $X$ as $\K{C}/\Lambda$ where $\Lambda$ is the lattice in $\K{C}$ that is generated by the two numbers $1$ and $\tau$ where $\tau$ is in the upper open half-plane. Without loss of generality we can further assume that the real part $X(\K{R})$ of $X$ is either empty, or $\K{R}/\Lambda$, or $(\K{R}\cup (\frac{1}{2}\tau+\K{R}))/\Lambda$.
 By \cite[Thm.~6.1]{milnor} the Julia set of $f$ is all of $X$ when the degree of $f$ is larger than $1$. So let us assume that the degree of $f$ is $1$. Since $f$ maps real points to real points, we can write $f$ as $$\K{C}/\Lambda\to \K{C}/\Lambda,\quad z\mapsto \alpha z+c,$$where $\alpha\in\{\pm1\}$. We first consider the case $\alpha=1$. Then $f^k(z)=z+k\cdot c$ which is either the identity map or has no fixed points. Thus either every point of $X$ is a periodic point or $f$ has no periodic points. If $\alpha=-1$, then we can apply the preceding argument to $f^2$.
\end{proof}

The situation is even easier for $g>1$. By the Riemann--Hurwitz formula \cite[Cor.~2.4]{Hart77} every nonconstant map $f:X\to X$ is an automorphism. But since by Hurwitz's theorem there are only finitely many automorphisms of $X$ \cite{hurwitz1}, we have that $f^k$ is the identity for large enough $k$. Therefore, we obtain the following proposition.

\begin{propo}\label{g>1}
  Let $g>1$. Then for a morphism $f: X\to X$ any point of $X$ is periodic.
\end{propo}

\section{Proofs of main results}

To prove Theorem \ref{thm:main} we need a couple of auxiliary facts that we prove first.
\begin{lemma}\label{lem:oneattr}
 Let $f\in\K{R}(z)_d$, $d\geq2$, be a rational function such that $J(f)\subsetneq\overline{\K{R}}$. Then $f$ has at most one attracting cycle. If further $J(f)$ is not connected, then all periodic points are real.
\end{lemma}

\begin{proof}
 The assumption implies that the Fatou set $F(f)=\overline{\K{C}}\setminus J(f)$ of $f$ is connected. Thus the basin of attraction of any attracting cycle of $f$ is all of $F(f)$ \cite[Cor.~4.12]{milnor}. In particular, there can be at most one attracting cycle of $f$. Now assume that there is a nonreal periodic point $z_0$, i.e., $f^k(z_0)=z_0$ for some $k\in\K{N}$. Since $J(f^k)=J(f)$ by \cite[Lem.~4.3]{milnor}, we can replace $f$ by $f^k$ and assume that $z_0$ is a nonreal fixed point of $f$. By assumption $z_0$ lies in the Fatou set of $f$. If $z_0$ was attracting, then its complex conjugate $\overline{z_0}$ would be an attracting fixed point as well which contradicts our first statement. Thus $|f'(z_0)|=1$ and by \cite[Lem.~11.1]{milnor} we have that $F(f)$ is conformally isomorphic to the unit disc. But this is a contradiction since $J(f)$ not being connected implies that $F(f)$ is not simply connected.
\end{proof}

\begin{lemma}\label{lem:blaschke}
 Let $f\in\K{R}(z)_d$ with $d\geq2$ be real fibered and assume that $f$ has a nonrepelling real cycle of length at most $2$. Then $f$ has only real periodic points.
\end{lemma}

\begin{proof}
 It suffices to show that for all $k\in\K{N}$ the function $g=f^{2k}$ has only real fixed points. By assumption $g$ maps the upper half-plane to itself and has a nonrepelling fixed point. Thus by \cite[\S16]{fatou} it has only real periodic points.
\end{proof}

We are now ready to prove Theorems \ref{thm:main} and \ref{thm:RF}.
\subsection{Proof of Theorem \ref{thm:main}}
 First assume $(1)$. Then by \cite[Thm. 14.1]{milnor} the Julia set $J(f)$ is real and we are in one of the cases of \Cref{thm:realjulia}. Assume we are in case $(i)$ and let $S=\overline{\K{R}}$. It is clear that $f^{-1}(S)\subset S$. If $f$ permutes upper and lower half-plane, then $f$ has at least $d$ real repelling fixed points by \cite[\S17]{fatou}, so $S$ contains a nonattracting fixed point. Furthermore, $f^2$ maps the upper half-plane to itself and thus has a nonrepelling real fixed point by \cite[\S16]{fatou}. If $f$ itself maps the upper half-plane to itself, it has a nonrepelling real fixed point for the same reason. If this fixed point is also nonattracting, we are done. If it is attracting, then there are additional $d$ repelling fixed points, again by \cite[\S16]{fatou}.
 
 If we are not in case $(i)$, then $J(f)\subsetneq\overline{\K{R}}$ so the Fatou set $F(f)$ is connected. The existence of a nonrepelling real fixed point is part of the statement of \Cref{thm:realjulia}. If this fixed point is attracting, then there is at least one other fixed point. This cannot be attracting by \Cref{lem:oneattr}. Thus $f$ has at least one real fixed point that is nonattracting. This is either a repelling or a parabolic fixed point (since real) and thus belongs to $J(f)$ by \cite[Lem.~4.6, 4.7]{milnor}. It was shown in \cite[p.~6454]{realmult} that there is a finite union of closed intervals $S\subset\overline{\K{R}}$ that contains $J(f)$ (and therefore a nonattracting fixed point) and further satisfies $f^{-1}(S)\subset S$.
 
Now assume $(2)$. Let $z_0\in S$ be a nonattracting fixed point. Since $z_0$ is real, it is either repelling or parabolic and thus belongs to $J(f)$ by \cite[Lem.~4.6, 4.7]{milnor}. By \cite[Cor.~4.13]{milnor} its iterated preimages are dense in $J(f)$. This implies that $J(f)$ is contained in $S$. If $J(f)=\overline{\K{R}}$, then $f$ is real fibered and by \Cref{lem:blaschke} $f$ has only real periodic points. If $J(f)\subsetneq\overline{\K{R}}$ and $J(f)$ is not connected, then $f$ has only real periodic points \Cref{lem:oneattr}. If $J(f)\subsetneq\overline{\K{R}}$ is connected, then it is a closed interval $I$. After conjugating with a M\"obius transformation, we can assume that $I=[0,\infty)$. Then according to \cite[\S25]{fatou} $f$ can be written as $$f(z)=z\cdot \left( c-\sum_{i=1}^m\frac{a_i}{z-b_i}\right)^2$$ for some nonnegative real numbers $a_i,b_i$ and $c$. A direct computation shows that such $f$ has only real fixed points. The same argument for $f^k$ shows that $f$ has only real periodic points.
 
 The additional statement follows because in the direction $(2)\Rightarrow(1)$ we have seen that such $S$ contains $J(f)$. Since all nonattracting periodic points are real and therefore repelling or parabolic, these are contained in $J(f)$ and thus in $S$.
\subsection{Proof of Theorem \ref{thm:RF}}
The only if direction is obvious. Observe that the real fibered function $g=f^2\in \K{R}(z)_{2d}$ maps the upper half-plane to itself. As $g$ has only real fixed points, \cite[\S16]{fatou} implies that it has a nonrepelling fixed point $z_0\in \overline{\K{R}}$. The $k$th iterate $g^{k}\in \K{R}(z)_{d^{2k}}$ of $g$ maps the upper half-plane to itself and $z_0$ is a nonrepelling fixed point also for $g^k$. Thus, again by \cite[\S16]{fatou}, $g^k$ has only real fixed points. If $f^{2\ell+1}$ had a nonreal fixed point, then so would also do $f^{2(2\ell+1)} = g^{2\ell+1}$.

  The condition for $f^2$ to have only real fixed points is a closed semialgebraic condition. By \cite[\S16]{fatou} such a function has a nonrepelling fixed point $z_0\in \overline{\K{R}}$. If $z_0$ is attracting, then $f^2$ has $d^2+1$ distinct real fixed points \cite[\S16]{fatou}. Any $h\in \K{R}(z)_d$ in a sufficiently small neighbourhood of $f$ is real fibered and $h^2$ has only real fixed points. Hence $f$ with an attracting periodic point lies in the interior of $\mathscr{R}_d$. If $z_0$ is an indifferent fixed point of $f^2$, then $z_0=f(z_0)$ must already be a fixed point of $f$. Indeed, if it was not the case, then  $z_0, f(z_0)\in \overline{\K{R}}$ would be two distinct indifferent fixed points of $f^2$. But \cite[\S16]{fatou} implies that real fibered rational functions can have at most one indifferent real fixed point.    
  We now show that the boundary of $\mathscr{R}_d$ consists precisely of functions $f$ with an indifferent fixed point. For this let us first recall the following fact: %, see \cite{} for details. 
  if $\phi: z\mapsto \frac{az+b}{cz+d}$, $ad-bc=1$, is a M\"obius transformation, then $\phi(z_0)$ is a periodic point for $f^\phi=\phi\circ f\circ\phi^{-1}$ if and only if $z_0$ is a periodic point for $f$. Moreover, the multiplier of $f^\phi$ at $\phi(z_0)$ equals the multiplier of $f$ at $z_0$. Therefore, after a conjugation of $f$ by a real M\"obius transformation $\phi$, we can assume that $f$ has an indifferent fixed point at $0$. The fixed points of $f=\frac{p}{q}$ are  the roots of the polynomial $F(z)=p(z)-zq(z)$ with the convention that $F(\infty)=0$ when $\deg(p)>\deg(q)$. In particular, all roots of $F$ are real and $0$ is the only multiple root of $F$ (otherwise, $f$ has more than one indifferent fixed point). We consider two cases.
  
  $i)$ The function $f$ maps the upper-half plane to itself. Then $f^\prime (0)=1$ and $0$ is a root of $F$ of multiplicity $2$ or $3$ (according to \cite[\S 16]{fatou} $0$ cannot have higher multiplicity). For all sufficiently small $\varepsilon \in \K{R}$ the function $f_\varepsilon = \frac{p_\varepsilon}{q}$, $p_\varepsilon = p+\varepsilon$, is real fibered and its fixed points are roots of $F_\varepsilon(z)=p_\varepsilon(z)-zq(z)= F(z)+\varepsilon$. If $0$ is a root of $F$ of multiplicity $2$ and $F^{\prime\prime}(0)>0$ (resp. $F^{\prime\prime}(0)<0$), then for all small $\varepsilon >0$ (resp. $\varepsilon<0$) the polynomial $F_\varepsilon$ has only $d-1<\deg(F)$ real roots. If $0$ is a root of $F$ of multiplicity $3$, for all small $\varepsilon$ the polynomial $F_\varepsilon$ has $d-1<\deg(F)$ real roots. Therefore, in both cases, an arbitrarily close function $f_\varepsilon$ to $f$ has nonreal fixed points.
  
  $ii)$ The function $f$ permutes upper and lower-half plane. Then $f^\prime(0)=-1$ and the second iterate $g=f^2=\frac{P}{Q}$ satisfies $g(0)=0$, $g^\prime(0)=(f^\prime(0))^2=1$ and $g^{\prime\prime}(0)=f^{\prime\prime}(0)(f^\prime(0))^2+f^{\prime\prime}(0)f^\prime(0)=0$. A direct computation shows that $0$ is a root of $G=P-zQ$ of multiplicity at least $3$, i.e., $G(0)=P(0)=Q(0)g(0)=0$, $G^\prime(0)= P^\prime(0)-Q(0)=Q(0)(g^\prime(0)-1)=0$ and $G^{\prime\prime}(0) =P^{\prime\prime}(0)-2Q^\prime(0)= Q(0)g^{\prime\prime}(0) =0$. By \cite[\S 16]{fatou} $0$ then has multiplicity exactly $3$. The remaining $d^2-2$ fixed points of $g$ are repelling again by \cite[\S 16]{fatou}. % Assume first that $f$ is strictly concave near $0$, that is, $f^{\prime\prime}(0)< 0$. Then there exists $\delta>0$ such that for all $z$ in $(-\delta,\delta)$ we have $f(z)\leq -z$ with equality only when $z=0$. Consider now a family of real fibered functions $f_\varepsilon = (1+\varepsilon)f$, $\varepsilon\geq 0$, with $f_0=f$ and fix positive $\varepsilon_0$ and $\delta_0<\delta$ such that $-\delta<(1+\varepsilon_0)f(z)<\delta$ for $z\in (-\delta_0, \delta_0)$. Note that $f$ is strictly decreasing on $(-\delta, \delta)$, $f(z)>0$ for $-\delta<z<0$ and $f(z)<0$ for $0<z<\delta$. Therefore, we have for any $\varepsilon\in [0,\varepsilon_0]$ that $f_\varepsilon(z)=(1+\varepsilon)f(z)>f(z)$ when $-\delta_0<z<0$, $f_\varepsilon(z)=(1+\varepsilon)f(z)<f(z)$ when $0<z<\delta_0$ and, using monotonicity of $f$,
  Note that there exists $\delta>0$ such that $f$ is strictly decreasing on $(-\delta,\delta)$. Consider a family of real fibered rational functions $f_\varepsilon=(1+\varepsilon)f$, $\varepsilon\geq 0$, with $f_0=f$ and fix positive $\varepsilon_0$ and $\delta_0<\delta$ such that $-\delta<(1+\varepsilon_0)f(z)<\delta$ for $z\in (-\delta_0,\delta_0)$. Since $f$ decreases on $(-\delta_0,\delta_0)$ and since $f(0)=0$, we have for any $\varepsilon\in (0,\varepsilon_0)$ that $(1+\varepsilon)f(z)>f(z)$ when $z\in (-\delta_0,0)$, $(1+\varepsilon)f(z)<f(z)$ when $z\in (0,\delta_0)$ and hence
  \begin{align*}
    g_\varepsilon(z)=f_\varepsilon^2(z) = (1+\varepsilon)f((1+\varepsilon)f(z))
    \begin{cases}
   &\hspace{-0.2cm} \leq (1+\varepsilon)f(f(z))=(1+\varepsilon)g,\quad \textrm{if}\ z\in(-\delta_0,0),\\  
    & \hspace{-0.2cm} \geq (1+\varepsilon)f(f(z)) = (1+\varepsilon)g,\quad \textrm{if}\ z\in(0,\delta_0).
\end{cases}
  \end{align*}
These inequalities mean that the graph of $g_\varepsilon$ over $(-\delta_0,0)$ (resp. $(0,\delta_0)$) lies strictly below (resp. above) the graph of $g$. In particular, $0$ is the only fixed point of $g_\varepsilon$, $\varepsilon\in (0,\varepsilon_0)$, in the interval $(-\delta_0,\delta_0)$, and this point is repelling since $g_\varepsilon^\prime(0) = (1+\varepsilon)^2(f^\prime(0))^2=(1+\varepsilon)^2>1$. Thus, for all sufficiently small $\varepsilon>0$ the function $g_\varepsilon$ has non-real fixed points as there are only $d^2-2$ fixed points outside $(-\delta_0,\delta_0)$. 

\section*{Acknowledgements}
We would like to thank Vladlen Timorin for useful comments and for bringing our attention to the work of Eremenko and van Strien.

\newcommand{\etalchar}[1]{$^{#1}$}
\def\cprime{$'$}

% \bibliographystyle{alpha}
 % \bibliography{biblio}

\begin{thebibliography}{BIJ{\etalchar{+}}19}

\bibitem[ASS17]{ASS}
H.~Abo, A.~Seigal, and B.~Sturmfels.
\newblock Eigenconfigurations of tensors.
\newblock {\em Algebraic and Geometric Methods in Discrete Mathematics,
  Contemporary Mathematics}, 685:1--25, 2017.

\bibitem[BIJ{\etalchar{+}}19]{survey}
R.~Benedetto, P.~Ingram, R.~Jones, M.~Manes, J.~H. Silverman, and T.~J. Tucker.
\newblock Current trends and open problems in arithmetic dynamics.
\newblock {\em Bulletin of the American Mathematical Society}, 56:611--685,
  2019.

\bibitem[EL89]{dynamics}
A.~E. Eremenko and M.~Yu. Lyubich.
\newblock The dynamics of analytic transformations.
\newblock {\em Algebra i Analiz}, 1(3):1--70, 1989.

\bibitem[EvS11]{realmult}
A.~E. Eremenko and S.~van Strien.
\newblock Rational maps with real multipliers.
\newblock {\em Trans. Amer. Math. Soc.}, 363(12):6453--6463, 2011.

\bibitem[Fat19]{fatou}
P.~Fatou.
\newblock Sur les \'{e}quations fonctionnelles.
\newblock {\em Bull. Soc. Math. France}, 47:161--271, 1919.

\bibitem[FS94]{FS}
J.E. Fornaess and N.~Sibony.
\newblock Complex dynamics in higher dimensions \textrm{I}.
\newblock {\em Ast\'{e}risque}, 222:201--231, 1994.

\bibitem[Har77]{Hart77}
R.~Hartshorne.
\newblock {\em Algebraic geometry}.
\newblock Springer-Verlag, New York-Heidelberg, 1977.
\newblock Graduate Texts in Mathematics, No. 52.

\bibitem[HSS01]{HSS}
J.~Hubbard, D.~Schleicher, and S.~Sutherland.
\newblock How to find all roots of complex polynomials by {N}ewton’s method.
\newblock {\em Invent. Math.}, 146:1--33, 2001.

\bibitem[Hur92]{hurwitz1}
A.~Hurwitz.
\newblock {\"U}ber algebraische {G}ebilde mit eindeutigen {T}ransformationen in
  sich.
\newblock {\em Math. Ann.}, 41(3):403--442, 1892.

\bibitem[Kal08]{Kalantari}
B.~Kalantari.
\newblock {\em Polynomial Root-Finding and Polynomiography}.
\newblock World Scientific Publishing Co., Inc., USA, 2008.

\bibitem[Koz18]{Kozhasov2018}
Kh. Kozhasov.
\newblock On fully real eigenconfigurations of tensors.
\newblock {\em SIAM J. Appl. Algebra Geom.}, 2(2):339--347, 2018.

\bibitem[KS20]{realfib}
M.~Kummer and E.~Shamovich.
\newblock Real fibered morphisms and {U}lrich sheaves.
\newblock {\em J. Algebraic Geom.}, 29(1):167--198, 2020.

\bibitem[Mil06]{milnor}
J.~Milnor.
\newblock {\em Dynamics in one complex variable}, volume 160 of {\em Annals of
  Mathematics Studies}.
\newblock Princeton University Press, Princeton, NJ, third edition, 2006.

\bibitem[MS94]{MS}
P.~Morton and J.~H. Silverman.
\newblock Rational periodic points of rational functions.
\newblock {\em Internat. Math. Res. Notices}, 2:97--110, 1994.

\bibitem[Nor50]{Northcott}
D.~G. Northcott.
\newblock Periodic points on an algebraic variety.
\newblock {\em Annals of Mathematics}, 51(1):167--177, 1950.

\end{thebibliography}
 \end{document}